\documentclass[11pt,twoside]{amsart}

\usepackage[utf8]{inputenc}
\usepackage[english]{babel}
\usepackage[T1]{fontenc}
\usepackage{amsfonts}
\usepackage{geometry}
\usepackage{pstricks, pst-node, pst-tree, pst-solides3d}
\usepackage{amsmath, amsthm, amssymb}
\usepackage{latexsym}
\usepackage{enumitem}
\usepackage{fancyhdr}
\usepackage{subcaption}
\usepackage{caption}
\usepackage{xcolor}
\usepackage{bm}
\usepackage[colorlinks=true, bookmarks=true]{hyperref}

\theoremstyle{plain}
\newtheorem{thm}{Theorem}[section]
\newtheorem{prop}[thm]{Proposition}
\newtheorem{cor}[thm]{Corollary}
\newtheorem{lem}[thm]{Lemma}
\theoremstyle{definition}
\newtheorem{defn}[thm]{Definition}
\newtheorem{ex}[thm]{Example}

\newtheorem{rem}[thm]{Remark}

\date{}

\usepackage{libertine}

\usepackage[scaled=.8]{beramono}

\newcommand{\N}{{\mathbb N}}
\newcommand{\hgt}{{\rm ht}\,}
\newcommand{\bight}{{\rm bight}\,}
\newcommand{\ara}{{\rm ara}\,}
\newcommand{\pd}{{\rm pd}}

\title[The arithmetical rank of the edge ideals of graphs with whiskers]{The arithmetical rank of the edge ideals \\ of graphs with whiskers}

\author{Antonio Macchia}
\address{\normalsize Antonio Macchia, Dipartimento di Matematica, Università degli Studi di Bari ``Aldo Moro'',\break Via Orabona 4, 70125 Bari, Italy}
\email{\normalsize \url{antonio.macchia@uniba.it}}

\begin{document}

\begin{abstract}
We consider the edge ideals of large classes of graphs with whiskers and for these ideals we prove that the arithmetical rank is equal to the big height. Then we extend these results to other classes of squarefree monomial ideals, generated in any degree, proving that the same equality holds.
\end{abstract}

\maketitle

\noindent {\bf Mathematics Subject Classification (2010):} 13A15, 13F55, 14M10, 05C05, 05C38.

\noindent {\bf Keywords:} Set-theoretic complete intersection ideals, arithmetical rank, edge ideals, whiskers, facet ideals.

\section{Introduction}

Given a Noetherian commutative ring with identity $R$, the \textit{arithmetical rank} (ara) of a proper ideal $I$ of $R$ is defined as the smallest integer $s$ for which there exist $s$ elements $a_1,\dots,a_s$ of $R$ such that the ideal $(a_1,\dots,a_s)$ has the same radical as $I$. In this case we will say that $a_1,\dots,a_s$ generate $I$ \textit{up to radical}. In general $\hgt I \leq \ara I$. If equality holds, $I$ is called a \textit{set-theoretic complete intersection}. As a consequence of the Auslander-Buchsbaum formula, whenever an ideal of $R=k[x_1,\dots,x_n]$ generated by squarefree monomials is a set-theoretic complete intersection, it is a Cohen-Macaulay ideal. The converse is not always true.
\\ We consider the case where $R$ is a polynomial ring over a field $K$ and $I$ is the so-called \textit{edge ideal} of a graph whose vertices are the indeterminates. Its set of generators is formed by the products of the pairs of indeterminates that form the edges of the graph. Thus $I$ is generated by squarefree monomials of degree 2, and is therefore a radical ideal. Large classes of graphs whose edge ideals are Cohen-Macaulay were described by Villarreal \cite{V90}. The arithmetical rank of edge ideals has recently been studied by several authors (see e.g. Kummini \cite{Ku09}) and explicitly determined for some special types of graphs.
\\ According to a well-known result by Lyubeznik \cite{L83}, if $I$ is a squarefree monomial ideal, the projective dimension of the quotient ring $R/I$, denoted $\pd_R\, R/I$, provides a lower bound for the arithmetical rank of $I$. We define the big height of $I$, denoted $\bight I$, as the maximum height of the minimal prime ideals of $I$. In general, we have $\hgt I \leq \bight I \leq \pd_R\, R/I \leq \ara I$. If $I$ is not unmixed, then $I$ is not a set-theoretic complete intersection, but it could still be true that $\bight I = \pd_R\, R/I = \ara I$. This equality has been established for the edge ideals of acyclic graphs (the so-called \textit{forests}) by Kimura and Terai \cite{KT13} (extending a result by Barile \cite{B08}). A weaker condition is the equality between the arithmetical rank and the projective dimension. This is the case for lexsegment edge ideals (see Ene, Olteanu, Terai \cite{EOT10}), for the graphs formed by one or two cycles connected through a path (\textit{cyclic} and \textit{bicyclic} graphs, see Barile, Kiani, Mohammadi and Yassemi \cite{BKMY12}) and for the graphs consisting of paths and cycles with a common vertex (see Kiani and Mohammadi \cite{KM12}). In all these cases, the arithmetical rank is independent of the field $K$.
\\ As a consequence of what we said above, the classes of Cohen-Macaulay monomial ideals are candidate to be set-theoretic complete intersections. We consider the family of \textit{whisker graphs}, obtained by adding a \textit{whisker} to each vertex of a given graph, i.e., by attaching a terminal edge to all its vertices. More in general, we can define the \textit{fully clique-whiskered graphs} in the following way: given a graph $G$, a subset $C$ of the vertex set $V(G)$ is a \textit{clique} if it induces a complete subgraph of $G$. If we partition $V(G)$ in cliques $W_1,\dots,W_t$ and add a new vertex $w_i$ for every clique and the edges $v w_i$ for every $v \in W_i$, then we call the resulting graph a \textit{fully clique-whiskered graph}. Cook and Nagel \cite{CN12} proved that the edge ideals of the fully clique-whiskered graphs are Cohen-Macaulay. In Section 3 we prove that the big height and the arithmetical rank are equal for a larger class of graphs with whiskers and, as a consequence, we will deduce that the edge ideals of the fully clique-whiskered graphs are set-theoretic complete intersections.
\\ The notion of whisker graph can be generalized in another direction. First we can consider a simplicial complex $\Delta$ on the vertex set formed by the indeterminates, instead of a graph, and define the \textit{facet ideal} as the ideal generated by the squarefree monomials corresponding to the facets of $\Delta$. Then we can add a simplex to each vertex of $\Delta$ and suppose that these simplices are pairwise disjoint. Faridi proved in \cite{F05} that the facet ideals obtained in this way are Cohen-Macaulay. In Section 4 we strengthen this result by showing that they are also set-theoretic complete intersections. Finally, we will add an arbitrary number of simplices to each vertex of $\Delta$ and prove that the big height of the facet ideal of the simplicial complex obtained in this way equals the arithmetical rank. All the results presented in this paper are independent of the field $K$.

\section{Preliminaries}

A useful technique that provides an upper bound for the arithmetical rank of ideals is the following result due to Schmitt and Vogel.

\begin{lem} \bf (\cite{SV79}, p. 249) \it \label{SchmittVogel}
Let $R$ be a commutative ring with identity and $P$ be a finite subset of elements of $R$. Let $P_0,\dots,P_r$ be subsets of $P$ such that
\begin{itemize}
\item[\textit{(i)}] $\bigcup_{i=0}^r P_i = P$;
\item[\textit{(ii)}] $P_0$ has exactly one element;
\item[\textit{(iii)}] if $p$ and $p'$ are different elements of $P_i$ ($0<i<r$), there is an integer $i'$, with $0 \leq i' < i$, and an element in $P_{i'}$ which divides $pp'$.
\end{itemize}
We set $q_i=\sum_{p \in P_i} p^{e(p)}$, where $e(p) \geq 1$ are arbitrary integers. We will write $(P)$ for the ideal of $R$ generated by the elements of $P$. Then
$$
\sqrt{(P)}=\sqrt{(q_0,\dots,q_r)}.
$$
\end{lem}

Another method to estimate the arithmetical rank of monomial ideals involves the Lyubeznik resolutions and was developed by Kimura in \cite{Ki09}.

Let $R=k[x_1,\dots,x_n]$ and $I$ a monomial ideal of $R$ with minimal set of generators $\left\{ u_1,\dots,u_{\mu} \right\}$, where $\mu = \mu(I)$ is the minimum number of generators of $I$. The \textit{Taylor resolution} of $I$ is
$$
T_{\bullet}: 0 \rightarrow T_{\mu} \stackrel{d_{\mu}}{\longrightarrow} T_{\mu-1} \stackrel{d_{\mu-1}}{\longrightarrow} \cdots \stackrel{d_1}{\longrightarrow} T_0 \rightarrow R/I \rightarrow 0,
$$
where $T_0 = Re_{\varnothing}$, $\displaystyle{T_s = \bigoplus_{1 \leq i_1 < \cdots < i_s \leq \mu} Re_{i_1 \cdots i_s}}$ and $e_{i_1 \cdots i_s}$ are free basis elements of $T_s$, with
$$
\deg\,e_{i_1 \cdots i_s} = \deg\,{\rm lcm}(u_{i_1},\dots,u_{i_s}).
$$
The differentials are defined by
$$
d_s(e_{i_1 \cdots i_s}) = \sum_{j=1}^s (-1)^{j-1} \frac{{\rm lcm}(u_{i_1},\dots,u_{i_s})}{{\rm lcm}(u_{i_1},\dots,\widehat{u_{i_j}},\dots,u_{i_s})} e_{i_1 \cdots \widehat{i_j} \cdots i_s}.
$$
A \textit{Lyubeznik resolution} is a graded free resolution of $R/I$ which is a subcomplex of the Taylor resolution of $I$.

\begin{defn} \bf (\cite{L88}) \rm
For every $1 \leq i_1 < \cdots < i_s \leq \mu$, the symbol $e_{i_1 \cdots i_s}$ is called \textit{$L$-admissible} if $u_q$ does not divide ${\rm lcm}(u_{i_t}, u_{i_{t+1}}, \dots, u_{i_s})$ for all $1 \leq t < s$ and for all $1 \leq q < i_t $. The \textit{Lyubeznik resolution} of $I$ is the subcomplex of the Taylor resolution of $I$ generated by all $L$-admissible symbols.
\end{defn}

In the following we will identify the symbol $e_{i_1 \cdots i_s}$ with the sequence of monomials $u_{i_1},\dots,u_{i_s}$.
\\ A Lyubeznik resolution of $I$ depends on the order of the generators $u_1, u_2, \dots, u_{\mu}$.

\begin{defn}
The \textit{$L$-length} $\lambda$ of $I$ is the minimum length of the Lyubeznik resolutions of $I$.
\end{defn}

While the Taylor resolution of $I$ is far from being a minimal graded free resolution of $I$, a Lyubeznik resolution of $I$ often is minimal.

\begin{thm} \bf (\cite{Ki09}, Theorem 1) \it \label{Kimura}
Let $I$ a monomial ideal of $R=k[x_1,\dots,x_n]$, then
$$
\ara I \leq \lambda.
$$
\end{thm}

In the following we will consider squarefree monomial ideals arising from graphs, the so-called \textit{edge ideals}.

\begin{defn}
Let $G$ be a graph with vertex set $V(G)= \{ x_1,\dots,x_n \}$, with $n \in \N$, $n \geq 1$, and whose edge set is $E(G)$. Suppose that $x_1,\dots,x_n$ are indeterminates over the field $K$. The \textit{edge ideal} of $G$ in the polynomial ring $R=K \left[ x_1,\dots,x_n \right]$ is the squarefree monomial ideal
$$
I(G)=\left( \left\{ x_i x_j \ \big|\ \{ x_i, x_j \} \in E(G) \right\} \right).
$$
For the sake of simplicity, we will use the same notation $x_i x_j$ for the monomial and for the corresponding edge.
\end{defn}

Let $G$ be a graph and $x$ a vertex of $G$. Adding a \textit{whisker} to the vertex $x$ of $G$ means adding a new vertex $y$ and the edge connecting $x$ and $y$.

\begin{defn}
A subset $C$ of $V(G)$ is a \textit{clique} if it induces a complete subgraph of $G$. A \textit{clique vertex-partition} of $G$ is a set $\pi = \{W_1,\dots,W_t\}$ of disjoint (possibly empty) cliques of $G$ whose union is $V(G)$.
\end{defn}

Notice that $G$ may admit many different clique vertex-partitions, and every graph has at least one clique vertex-partition, namely the trivial partition $\tau = \{\{x_1\},\dots,\{x_n\}\}$.

\begin{defn}
Given a clique $W$ of $G$, a \textit{clique-whiskering} of $W$ is given by adding a new vertex $w$ and connecting $w$ to every vertex in $W$. Let $\pi = \{ W_1,\dots,W_t \}$ be a clique vertex-partition of $G$. Consider the clique-whiskering of every clique of $\pi$ obtained by adding the vertex $w_i$ to $W_i$ (where $w_i \neq w_j$ if $i \neq j$). We call the graph $G^{\pi}$ obtained in this way \textit{fully clique-whiskered}. This graph has vertex set $V(G) \cup \{ w_1,\dots,w_t \}$ and edge set $E(G) \cup \{ vw_i\ |\ v \in W_i \}$.
\end{defn}

If $\tau$ is the trivial partition, we call the $G^{\tau}$ the \textit{whisker graph on} $G$. Note that empty cliques produce isolated vertices.

\begin{ex}
Let $G$ be the three-cycle $C_3$ on the vertex set $\{x_1,x_2,x_3\}$. There are three distinct clique vertex-partitions of $G$ (without empty cliques): the trivial partition $\tau = \{\{x_1\},\{x_2\},\{x_3\}\}$, $\pi = \{ \{x_1,x_2\}, \{x_3\}\}$ and $\rho = \{\{x_1,x_2,x_3\}\}$. These partitions produce the following fully clique-whiskered graphs:
\clearpage
\begin{figure}[ht!]
\begin{subfigure}[t]{0.27\textwidth}
\centering
\psset{unit=1.2cm}
\begin{pspicture}(-0.3,-1.3)(2.5,2.5)
\psline(0,0)(1.5,0)
\psline(1.5,0)(0.75,1.3)
\psline(0.75,1.3)(0,0)
\psline(0.75,1.3)(0.75,2.3)
\psline(1.5,0)(2.37,-0.5)
\psline(0,0)(-0.87,-0.5)
\psdots(0,0)
\psdots(1.5,0)
\psdots(0.75,1.3)
\psdots(0.75,2.3)
\psdots(2.37,-0.5)
\psdots(-0.87,-0.5)
\begin{small}
\uput[150](0,0){$x_1$}
\uput[30](1.5,0){$x_2$}
\uput[0](0.75,1.3){$x_3$}
\uput[180](-0.87,-0.5){$y_1$}
\uput[0](2.37,-0.5){$y_2$}
\uput[0](0.75,2.3){$y_3$}
\rput(-0.5,2){$G^{\tau}$}
\end{small}
\end{pspicture}
\end{subfigure}\hspace{7mm}\begin{subfigure}[t]{0.27\textwidth}
\centering
\psset{unit=1.2cm}
\begin{pspicture}(-0.3,-1.3)(2.5,2.5)
\psline(0,0)(1.5,0)
\psline(1.5,0)(0.75,1.3)
\psline(0.75,1.3)(0,0)
\psline(0.75,1.3)(0.75,2.3)
\psline(1.5,0)(0.75,-1)
\psline(0,0)(0.75,-1)
\psdots(0,0)
\psdots(1.5,0)
\psdots(0.75,1.3)
\psdots(0.75,2.3)
\psdots(0.75,-1)
\begin{small}
\uput[210](0,0){$x_1$}
\uput[-30](1.5,0){$x_2$}
\uput[0](0.75,1.3){$x_3$}
\uput[270](0.75,-1){$y_1$}
\uput[0](0.75,2.3){$y_2$}
\rput(-0.5,2){$G^{\pi}$}
\end{small}
\end{pspicture}
\end{subfigure}\hspace{3mm}\begin{subfigure}[t]{0.27\textwidth}
\centering
\psset{unit=1.2cm}
\begin{pspicture}(-0.3,-1.3)(2.5,2.5)
\psline(0,0)(2,0)
\psline(2,0)(1,1.73)
\psline(1,1.73)(0,0)
\psline(1,1.73)(1,0.58)
\psline(2,0)(1,0.58)
\psline(1,0.58)(0,0)
\psdots(0,0)
\psdots(2,0)
\psdots(1,1.73)
\psdots(1,0.58)
\begin{small}
\uput[210](0,0){$x_1$}
\uput[-30](2,0){$x_2$}
\uput[90](1,1.73){$x_3$}
\uput[270](1,0.58){$y_1$}
\rput(-0.3,2){$G^{\rho}$}
\end{small}
\end{pspicture}
\end{subfigure}
\caption{Clique-whiskerings of the three-cycle}
\end{figure}
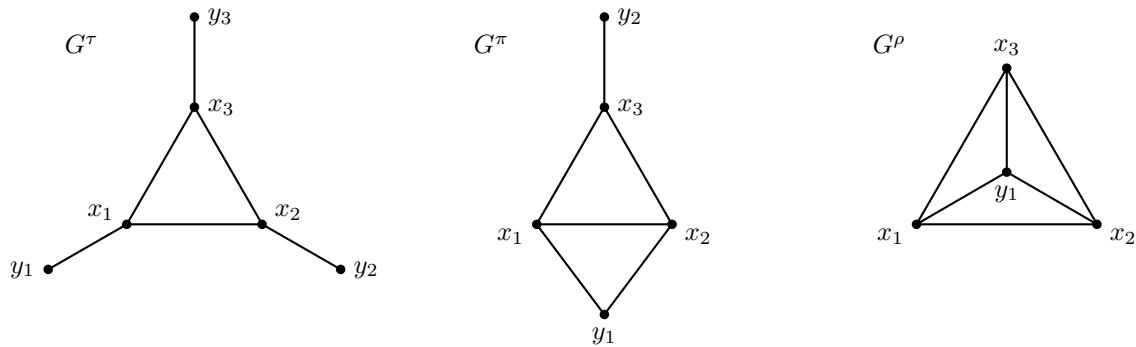
\end{ex}

Cook and Nagel \cite{CN12} have shown the following result:

\begin{thm} \bf (\cite{CN12}, Corollary 3.5) \it \label{cliquewhiskering}
Let $\pi$ be a clique vertex-partition of a graph $G$ and let $G^{\pi}$ be the fully clique-whiskering graph of $G$ on $\pi$. Then the ideal $I(G^{\pi})$ is Cohen-Macaulay.
\end{thm}

\begin{rem} \label{CookNagel}
Suppose that $|V(G)|=n$. Then $\hgt I(G^{\pi})=n$ because $I(G^{\pi})$ is pure (see Bruns-Herzog \cite{BH93}, Cor. 5.1.5) and the ideal generated by the vertices of $G$ is a minimal prime ideal of $I(G^{\pi})$.
\\ Theorem \ref{cliquewhiskering} had previously been proven by Dochtermann and Engstr\"om \cite{DE09} for whisker graphs, which is also a special case of \cite{F05}, Theorem 8.2.
\end{rem}

In the same way as squarefree monomial ideals generated in degree two can be attached to graphs, squarefree monomial ideals with generators of any degree can, more in general, be attached to simplicial complexes. This gives rise to the notion of facet ideal, which has been extensively studied by Faridi in \cite{F02} and \cite{F05}.

\begin{defn}
Let $\Delta$ be a simplicial complex with vertex set $V(\Delta)$ and facets $F_1,\dots,F_q$. A \textit{vertex cover} for $\Delta$ is a subset $A$ of $V(\Delta)$, with the property that for every facet $F_i$ there is a vertex $v \in A$ such that $v \in F_i$. A \textit{minimal vertex cover} of $\Delta$ is a subset $A$ of $V(\Delta)$ such that $A$ is a vertex cover and no proper subset of $A$ is a vertex cover for $\Delta$.
\end{defn}

\begin{defn}
Let $\Delta$ be a simplicial complex on the vertex set $V(\Delta)=\{ x_1,\dots,x_n \}$. The \textit{facet ideal} of $\Delta$ is the ideal $I(\Delta)$ in $k[x_1,\dots,x_n]$ generated by all squarefree monomials $x_{i_1} \cdots x_{i_s}$, such that $\{ x_{i_1},\dots,x_{i_s} \}$ is a facet of $\Delta$.
\end{defn}

\begin{prop} \bf (\cite{F02}, Proposition 1) \it \label{Faridi}
Let $\Delta$ be a simplicial complex over $n$ vertices. Consider the facet ideal $I(\Delta)$ in the polynomial ring $R=k[x_1,\dots,x_n]$. Then an ideal $P=(x_{i_1},\dots,x_{i_s})$ of $R$ is a minimal prime of $I(\Delta)$ if and only if $\{ x_{i_1},\dots,x_{i_s} \}$ is a minimal vertex cover for $\Delta$.
\end{prop}

Let us consider a simplicial complex $\Delta$ on the vertex set $V(\Delta) = \{ x_1,\dots,x_n \}$. For every vertex $x_i$, we add a facet $F_i$ of dimension $\geq 1$ such that
\begin{itemize}
\item $F_i \cap V(\Delta) = \{ x_i \}$, for every $i=1,\dots,n$,
\item $F_i \cap F_j = \varnothing$ if $i \neq j$, for every $i,j=1,\dots,n$.
\end{itemize}
We will call $\Delta'$ the simplicial complex obtained in this way.

\begin{thm} \bf (\cite{F05}, Theorem 8.2) \it
With respect to the above notations, the facet ideal $I(\Delta')$ is Cohen-Macaulay.
\end{thm}

\begin{rem} \label{Faridi2}
We have that $\hgt I(\Delta') = n$ because $I(\Delta')$ is pure (see Bruns, Herzog \cite{BH93}, Cor. 5.1.5) and, in view of Proposition \ref{Faridi}, the ideal generated by the vertices of $\Delta$ is a minimal prime ideal of $I(\Delta')$.
\end{rem}

\section{The arithmetical rank of the edge ideals of graphs with whiskers}

In this section we prove the equality between the arithmetical rank and the (big) height for the edge ideals of some classes of graphs obtained by adding one or more whiskers to every vertex of a given graph. We provide explicit formulas and examine the special case where the graph is a cycle.

\begin{prop} \label{P1.whiskergraph}
Let $G$ be a graph on the vertex set $V(G)=\{ x_1,\dots,x_n \}$. Consider a partition $\{W_1,\dots,W_t\}$ of $V(G)$. For all $i=1,\dots,t$, and for every $x_j \in W_i$ add a new vertex $y_i$ and the whisker $x_j y_i$. Let $G'$ be the graph obtained in this way. Then
$$
\bight I(G') = \ara I(G') = n.
$$
In particular, every fully clique-whiskered graph is a set-theoretic complete intersection.
\end{prop}

\begin{proof}
Suppose that $W_1 = \{ x_1,\dots,x_{m_1} \}$, $W_2 = \{ x_{m_1+1},\dots,x_{m_2} \},\dots$, $W_t = \{ x_{m_{t-1}+1},\dots,x_{m_t} = x_n \}$. First we show that $\bight I(G') \geq n$. The set $A = \{ x_1,\dots,x_n \}$ is a minimal vertex cover since one of the vertices of each edge of $G'$ belongs to $A$. The vertex cover $A$ is minimal because removing a vertex $x_j \in W_i$, for some $i,j$, would leave the whisker $x_j y_i$ uncovered. Hence $\bight I(G') \geq n$.
\\ Next, we prove that $\ara I(G') \leq n$. Let us consider the following ordering of the quadratic monomials:
\begin{center}
\begin{tabular}[c]{cccccc}
$x_1x_2$ & $x_1x_3$ & $\cdots$                      & $\cdots$ & $x_1x_n$            & $x_1y_1$ \\
         & $\ddots$ &                               &          & $\vdots$            & $\vdots$ \\
         &          & $x_{m_1} x_{m_1+1}$           & $\cdots$ & $x_{m_1} x_n$       & $x_{m_1} y_1$ \\
         &          & $x_{m_1+1} x_{m_1+2}$         & $\cdots$ & $x_{m_1+1} x_n$     & $x_{m_1+1} y_2$ \\
         &          & $\vdots$                      &          & $\vdots$            & $\vdots$ \\
         &          & $x_{m_{t-1}+1} x_{m_{t-1}+2}$ & $\cdots$ & $x_{m_{t-1}+1} x_n$ & $x_{m_{t-1}+1} y_t$ \\
         &          &                               &          & $\ddots$            & $\vdots$ \\
         &          &                               &          &                     & $x_ny_t$
\end{tabular}
\end{center}
and arrange the generators of $I(G')$ according to the induced ordering. Let $u$ be an admissible symbol for $I(G')$. We want to show, by induction on $n$, that $u$ has length at most $n$. The claim is true for $n=2$, because the symbol consisting of $x_1x_2, x_1y_1, x_2y_i$ is not admissible if $i=1$ or $i=2$. Let $n>2$ and suppose that there are exactly $r$ monomials in $u$ containing the variable $x_1$ (these are monomials appearing in the first row of the above table), and precisely
$$
x_1x_{i_1}, x_1x_{i_2}, \dots, x_1x_{i_r}, \text{\qquad or \qquad} x_1x_{i_1}, x_1x_{i_2}, \dots, x_1x_{i_{r-1}}, x_1y_1.
$$
In both cases, every other monomial in $u$ cannot contain any of the variables $x_{i_1}, x_{i_2}, \dots, x_{i_{r-1}}$, because $x_1x_{i_j}$ divides $x_1x_{i_r}x_{i_j}$ and $x_1y_1x_{i_j}$ for all $j = 1,\dots,r-1$. Thus, the remaining monomials in $u$ form an admissible symbol for a graph of the same type of $G'$ on a vertex set $W \subset \left(\{x_2,\dots,x_n\} \smallsetminus \{x_{i_1},\dots,x_{i_{r-1}}\}\right) \cup \{y_1,\dots,y_t\}$. By induction, this symbol has length at most $n-1-(r-1)=n-r$. Therefore $u$ has length at most $r+n-r=n$. It follows, from Theorem \ref{Kimura}, that $\ara I(G') \leq n$. Hence $\bight I(G') = \ara I(G') = n$.
\\ The second part of the statement follows from Remark \ref{CookNagel}.
\end{proof}

If $G$ is the whisker graph on a cycle graph $C_n$, using Lemma \ref{SchmittVogel} we can find $n$ polynomials that generate $I(G)$ up to radical and whose expressions are simpler than those obtained using the technique due to Kimura (compare what follows with the proof of Theorem 1 in \cite{Ki09}).
\\ Given a cycle $C_n$ on the vertex set $V(C_n)=\{x_1,\dots,x_n\}$, we consider the $n$-sunlet graph $S_n$ on $C_n$, obtained by adding to each vertex $x_i$ of $C_n$ a whisker, whose terminal vertex is $y_i$, for all $i=1,\dots,n$.

\begin{ex} \label{sunlet}
For each $n \in \N$, $n \geq 3$, the edge ideal of the $n$-sunlet graph $S_n$ is a set-theoretic complete intersection, namely
$$
\ara I(S_n) = \hgt I(S_n) = |V(C_n)| = n.
$$
We distinguish the following cases.
\\ If $n=3$, consider the following sums of monomials
\begin{align*}
q_0 &= x_1 x_2 \\
q_1 &= x_1 x_3 + x_2 x_3 \\
q_2 &= x_1 y_1 + x_2 y_2 + x_3 y_3.
\end{align*}
If $n=4$, set
\begin{align*}
q_0 &= x_1 x_2 \\
q_1 &= x_1 x_4 + x_2 x_3 \\
q_2 &= x_1 y_1 + x_2 y_2 + x_3 x_4 \\
q_3 &= x_3 y_3 + x_4 y_4.
\end{align*}
Finally, for $n=5$, set
\begin{align*}
q_0 &= x_1 x_2 \\
q_1 &= x_1 x_5 + x_2 x_3 \\
q_2 &= x_1 y_1 + x_4 x_5 \\
q_3 &= x_2 y_2 + x_3 x_4 + x_3 y_3 x_5 y_5 \\
q_4 &= x_3 y_3 + x_4 y_4 + x_5 y_5.
\end{align*}
Now suppose that $n \geq 6$. In this case set
\begin{align*}
q_0 &= x_1 x_2 \\
q_1 &= x_1 x_n + x_2 x_3 \\
q_2 &= x_2 y_2 + x_3 x_4 \\
\vdots & 
\end{align*}
\begin{align*}
q_{n-4} &= x_{n-4} y_{n-4} + x_{n-3} x_{n-2} \\
q_{n-3} &= x_1 y_1 + x_{n-1} x_n \\
q_{n-2} &= x_{n-3} y_{n-3} + x_{n-2} x_{n-1} + x_{n-2} y_{n-2} x_n y_n \\
q_{n-1} &= x_{n-2} y_{n-2} + x_{n-1} y_{n-1} + x_n y_n.
\end{align*}
\\ Then, in any case, we have $I(S_n) = \sqrt{(q_0,\dots,q_{n-1})}$ by Lemma \ref{SchmittVogel}. We show that its assumptions are fulfilled by the sets $P_0,\dots,P_{n-1}$, where, for all $i=0,\dots,n-1$, $P_i$ is the set of monomials appearing in $q_i$. It is straightforward to verify that conditions $\textit{(i)}$ and $\textit{(ii)}$ are satisfied. Evidently condition $\textit{(iii)}$ is true if $n \in \{3,4,5\}$. We prove it for $n \geq 6$. The product of the monomials in $P_1$ is $x_1 x_n \cdot x_2 x_3$, which is a multiple of $x_1 x_2 \in P_0$. For $i=2,\dots,n-4$, the product of the monomials of $P_i$ is $x_i y_i \cdot x_{i+1} x_{i+2}$, which is a multiple of $x_i x_{i+1} \in P_{i-1}$. The product of the monomials of $P_{n-3}$ is $x_1 y_1 \cdot x_{n-1} x_n$, a multiple of $x_1 x_n \in P_1$. In $P_{n-2}$, we can form three products: $x_{n-3} y_{n-3} \cdot x_{n-2} x_{n-1}$ and $x_{n-3} y_{n-3} \cdot x_{n-2} y_{n-2} x_n y_n$, which are multiples of $x_{n-3} x_{n-2} \in P_{n-4}$, and $x_{n-2} x_{n-1} \cdot x_{n-2} y_{n-2} x_n y_n$, which is a multiple of $x_{n-1} x_n \in P_{n-3}$. As for $P_{n-1}$, we have $x_{n-2} y_{n-2} \cdot x_{n-1} y_{n-1}$, which is a multiple of $x_{n-2} x_{n-1} \in P_{n-2}$, $x_{n-2} y_{n-2} \cdot x_n y_n$, which is an element of $P_{n-2}$, and $x_{n-1} y_{n-1} \cdot x_n y_n$ which is a multiple of $x_{n-1} x_n \in P_{n-3}$. This completes the proof.
\end{ex}

Example \ref{sunlet} can be generalized as follows.

\begin{ex} \label{multiwhisker}
Let $C_n$ be a cycle graph with vertex set $\{ x_1,\dots,x_n \}$ and add $k$ whiskers to each vertex of $C_n$. Let $G$ be the graph obtained in this way and $y_{i,1},\dots,y_{i,k}$ be the terminal vertices of the whiskers on $x_i$. Then
$$
\bight I(G) = \ara I(G) = \left\lceil \frac{n}{2} \right\rceil + \left\lfloor \frac{n}{2} \right\rfloor k.
$$
First we define a minimal vertex cover of $G$. We choose the vertices
$$
x_{2j-1} \text{ for all } j=1,\dots, \left\lceil \frac{n}{2} \right\rceil \text{ and } y_{2j,1},\dots,y_{2j,k} \text{ for all } j=1,\dots, \left\lfloor \frac{n}{2} \right\rfloor.
$$
Thus
$$
\bight I(G) \geq \left\lceil \frac{n}{2} \right\rceil + \left\lfloor \frac{n}{2} \right\rfloor k.
$$
Now we define the same number of polynomials generating $I(G)$ up to radical. We distinguish three cases:
\begin{itemize}
\item if $n=3$, we set
    \begin{align*}
    q_0 &= x_1 x_2 \\
    q_1 &= x_1 x_3 + x_2 x_3 \\
    q_{j+1} &= x_1 y_{1,j} + x_2 y_{2,j} + x_3 y_{3,j}
    \end{align*}
    for $j \in \{ 1,\dots,k \}$. Thus, by Lemma \ref{SchmittVogel}, $\ara I(G) \leq 2 + k = \left\lceil \frac{3}{2} \right\rceil + \left\lfloor \frac{3}{2} \right\rfloor k$.
\item if $n$ is even, let $q_0,\dots,q_{n-1}$ be as in Example \ref{sunlet}, with $y_j$ replaced by $y_{j,1}$ for all $j=1,\dots,n$. Then set
    \begin{align*}
    q_{n+\frac{n}{2} (j-2)} &= x_1 y_{1,j} + x_2 y_{2,j} \\
    q_{n+\frac{n}{2} (j-2)+1} &= x_3 y_{3,j} + x_4 y_{4,j} \\
    \vdots & \\
    q_{n+\frac{n}{2} (j-2)+ \frac{n}{2}-1} &= x_{n-1} y_{n-1,j} + x_n y_{n,j}
    \end{align*}
    for $j \in \{ 2,\dots,k \}$. According to Lemma \ref{SchmittVogel}, the polynomials $q_0,\dots,q_{n+\frac{n}{2}(k-1)-1}$ generate $I(G)$ up to radical. Hence $\ara I(G) \leq n + \frac{n}{2} (k-1) = \left\lceil \frac{n}{2} \right\rceil + \left\lfloor \frac{n}{2} \right\rfloor k$.
\item if $n$ is odd, let $q_0,\dots,q_{n-1}$ be as in Example \ref{sunlet}, with $y_j$ replaced by $y_{j,1}$ for all $j=1,\dots,n$. Then set
    \begin{align*}
    q_{n + \left\lfloor \frac{n}{2} \right\rfloor (j-2)} &= x_1 y_{1,j} + x_2 y_{2,j} + x_n y_{n,j} x_3 y_{3,j} \\
    q_{n + \left\lfloor \frac{n}{2} \right\rfloor (j-2)+1} &= x_3 y_{3,j} + x_4 y_{4,j} + x_n y_{n,j} x_5 y_{5,j} \\
    \vdots & \\
    q_{n + \left\lfloor \frac{n}{2} \right\rfloor (j-2)+ \left\lfloor \frac{n}{2} \right\rfloor-2} &= x_{n-4} y_{n-4,j} + x_{n-3} y_{n-3,j} + x_n y_{n,j} x_{n-2} y_{n-2,j} \\
    q_{n + \left\lfloor \frac{n}{2} \right\rfloor (j-2)+ \left\lfloor \frac{n}{2} \right\rfloor-1} &= x_{n-2} y_{n-2,j} + x_{n-1} y_{n-1,j} + x_n y_{n,j}
    \end{align*}
    for $j \in \{ 2,\dots,k \}$. By Lemma \ref{SchmittVogel}, $\ara I(G) \leq n + \left\lfloor \frac{n}{2} \right\rfloor (k-1) = \left\lceil \frac{n}{2} \right\rceil + \left\lfloor \frac{n}{2} \right\rfloor k$.
\end{itemize}
Therefore
\[
\bight I(G) = \ara I(G) = \left\lceil \frac{n}{2} \right\rceil + \left\lfloor \frac{n}{2} \right\rfloor k.\qedhere
\]
\end{ex}

\section{The arithmetical rank of some facet ideals}

In this section we prove the equality between the arithmetical rank and the (big) height for the facet ideals of some classes of simplicial complexes obtained by adding one or more facets to every vertex of a given simplicial complex. Again, we provide explicit formulas. The simplicial complexes considered in the following proposition are a special case of the so-called grafted simplicial complexes considered by Faridi in \cite{F05}.

\begin{prop} \label{P2.whiskersimplcomplex}
Let $\Delta$ be a simplicial complex on the vertex set $V(\Delta) = \{ x_1,\dots,x_n \}$. For every vertex $x_i$, we add a facet $F_i$ of dimension $\geq 1$ such that
\begin{itemize}
\item $F_i \cap V(\Delta) = \{ x_i \}$, for every $i=1,\dots,n$,
\item $F_i \cap F_j = \varnothing$ if $i \neq j$, for every $i,j=1,\dots,n$.
\end{itemize}
Let $\Delta'$ be the simplicial complex obtained in this way. Then
$$
\hgt I(\Delta') = \ara I(\Delta') = n,
$$
thus $I(\Delta')$ is a set-theoretic complete intersection.
\end{prop}

\begin{proof}
From Remark \ref{Faridi2}, it follows that $\hgt I(\Delta') = n$. we show that $\ara I(\Delta') \leq n$ by proving that the $L$-length of $I(\Delta')$ is at most $n$. For every facet $F_i$, we denote by $u_{F_i}$ the monomial corresponding to $F_i$ and set $F_i = \left\{ x_i, y^{(i)}_1, \dots, y^{(i)}_{p_i} \right\}$. Let us consider the lexicographic ordering of the monomial generators of $I(\Delta')$, with $x_1 < x_2 < \cdots < x_n < y^{(1)}_1 < \cdots < y^{(1)}_{p_1} < \cdots < y^{(n)}_1 < \cdots < y^{(n)}_{p_n}$:
\begin{center}
\begin{tabular}[c]{ccccc}
$u_{1,1}$ & $u_{1,2}$ & $\cdots$ & $u_{1,s_1}$  & $u_{F_1}$ \\
          & $u_{2,1}$ & $\cdots$ & $u_{2,s_2}$  & $u_{F_2}$ \\
          &           &          & $\ddots$     & $\vdots$ \\
          &           &          &              & $u_{F_n}$
\end{tabular}
\end{center}
Let $\alpha$ be an admissible symbol for $\Delta'$ consisting of the monomials $\alpha_1,\dots,\alpha_m$, for some $m \geq 1$. We want to show, by induction on $n$, that $\alpha$ has length at most $n$. The claim is true for $n=2$, because the symbol consisting of $x_1x_2, u_{F_1}, u_{F_2}$ is not admissible. Let $n>2$ and suppose that there are exactly $r$ monomials in $\alpha$ containing the variable $x_1$ (these are monomials appearing in the first row of the above table). If $r=0$, then $\alpha$ is an admissible symbol on the vertex set $\{ x_2,\dots,x_n \} \cup \left\{y^{(i)}_1, \dots, y^{(i)}_{p_i}\ \big|\ i=2,\dots,n \right\}$, so the claim follows by induction. If $r=1$, then the monomials $\alpha_2, \dots, \alpha_m$ form an admissible symbol $\beta$ on the vertex set $\{ x_2,\dots,x_n \} \cup \left\{y^{(i)}_1, \dots, y^{(i)}_{p_i}\ \big|\ i=2,\dots,n \right\}$ and, by induction, $|\beta| \leq n-1$. Hence $|\alpha| = 1 + |\beta| \leq 1+n-1 = n$. Now suppose that $r \geq 2$ and that the monomials of $\alpha$ containing the variable $x_1$ are precisely
$$
\alpha_1=u_{1,i_1}, \alpha_2=u_{1,i_2}, \dots, \alpha_{r-1}=u_{1,i_{r-1}}, \text{ and } \alpha_r=u_{1,i_r} \text{ or } \alpha_r=u_{F_1}.
$$
We set $V_i = \{ x_j\ |\ x_j \text{ divides } \alpha_i,\ x_j \neq x_1 \}$ for every $i=1,\dots,r$ (if $\alpha_r = u_{F_1}$, then $V_r = \varnothing$). Consider, for every $i = 1,\dots, r-1$, the set $V_i \smallsetminus W_i$, where $W_i = V_{i+1} \cup \cdots \cup V_r$.
\\ The following two properties hold:
\begin{itemize}
\item[1)] $V_i \smallsetminus W_i \neq \varnothing$, for every $i = 1,\dots,r-1$,
\item[2)] $(V_i \smallsetminus W_i) \cap V_j = \varnothing$, for every $1 \leq i < j \leq r-1$.
\end{itemize}
For the first property, suppose for a contradiction that $V_i \smallsetminus W_i = \varnothing$ for some $i$. Then $V_i \subset W_i = V_{i+1} \cup \cdots \cup V_r$ and this implies that $\alpha_i$ divides ${\rm lcm}(\alpha_{i+1},\dots,\alpha_r)$, against the assumption that $\alpha$ is admissible.
\\ The second property is true because $V_j \subset W_i$ if $1 \leq i < j \leq r-1$.
\\ Note that the indeterminates in $V_i \cap W_i$ appear both in $\alpha_i$ and in some of the monomials $\alpha_{i+1},\dots,\alpha_r$. Therefore, for the symbol $\alpha$ to be admissible, at least one indeterminate in $V_i \smallsetminus W_i$ must not appear in the monomials $\alpha_{r+1},\dots,\alpha_m$ (otherwise $\alpha_i$ divides ${\rm lcm}(\alpha_{i+1},\dots,\alpha_m)$). By virtue of 2), for every $j = 1,\dots,r-1$, we can thus choose an indeterminate $x_{i_j}$ that appears in $\alpha_j$ and does not appear in the monomials $\alpha_{r+1},\dots,\alpha_m$, in such a way that $x_{i_1},\dots,x_{i_{r-1}}$ are pairwise distinct. It follows that the monomials $\alpha_{r+1},\dots,\alpha_m$ form an admissible symbol $\beta$ on the vertex set $\left(\{ x_2,\dots,x_n\} \smallsetminus \{ x_{i_1},\dots,x_{i_{r-1}} \}\right) \cup \left\{y^{(i)}_1, \dots, y^{(i)}_{p_i}\ \big|\ i \in \{2,\dots,n\} \smallsetminus \{ i_1,\dots,i_{r-1} \} \right\}$. By induction, $|\beta| \leq n-1 - (r-1) = n-r$. Then $|\alpha| = r + |\beta| \leq r+n-r = n$.
\end{proof}

\begin{cor} \label{multifacet}
Let $\Delta$ be a simplicial complex on the vertex set $\{ x_1,\dots,x_n \}$. For every vertex $x_i$, we add $m_i \geq 1$ facets $F_{i,1},\dots, F_{i,m_i}$ of dimension $\geq 1$ such that
\begin{itemize}
\item $F_{i,j} \cap V(\Delta) = \{ x_i \}$, for every $i=1,\dots,n$ and $j=1,\dots,m_i$,
\item $F_{i,j} \cap F_{i,k} = \{ x_i \}$, for every $i=1,\dots,n$ and $j,k=1,\dots,m_i$, $j \neq k$,
\item $F_{i,j} \cap F_{h,k} = \varnothing$ if $i \neq h$, for every $i,h=1,\dots,n$, $j=1,\dots,m_i$ and $k=1,\dots,m_h$.
\end{itemize}
Call $\Delta'$ the simplicial complex obtained in this way. Then
$$
\bight I(\Delta') = \ara I(\Delta').
$$
\end{cor}

\begin{proof}
For every facet $F_{i,j}$, we denote by $u_{F_{i,j}}$ the monomial corresponding to $F_{i,j}$ and suppose that\break$F_{i,j} = \left\{ x_i, y^{(i,j)}_1, \dots, y^{(i,j)}_{p_{(i,j)}} \right\}$. Let us consider the lexicographic ordering of the monomial generators of $I(\Delta')$, with
$$
x_1 < x_2 < \cdots < x_n < y^{(1,1)}_1 < \cdots < y^{(1,1)}_{p_{(1,1)}} < \cdots < y^{(1,m_1)}_1 < \cdots < y^{(1,m_1)}_{p_{(1,m_1)}} < \cdots < y^{(n,m_n)}_1 < \cdots < y^{(n,m_n)}_{p_{(n,m_n)}}:
$$
\begin{center}
\begin{tabular}[c]{ccccccc}
$u_{1,1}$ & $u_{1,2}$ & $\cdots$ & $u_{1,s_1}$  & $u_{F_{1,1}}$   & $\cdots$ & $u_{F_{1,m_1}}$ \\
          & $u_{2,1}$ & $\cdots$ & $u_{2,s_2}$  & $u_{F_{2,1}}$   & $\cdots$ & $u_{F_{2,m_2}}$ \\
          &           & $\ddots$ & $\vdots$     & $\vdots$        &          & $\vdots$ \\
          &           &          & $x_{n-1}x_n$ & $u_{F_{n-1,1}}$ & $\cdots$ & $u_{F_{n-1,m_{n-1}}}$ \\
          &           &          &              & $u_{F_{n,1}}$   & $\cdots$ & $u_{F_{n,m_n}}$
\end{tabular}
\end{center}
Let $\alpha$ be an admissible symbol for $\Delta'$ with maximal length and call $\Gamma$ the simplicial complex $\Delta \cup \langle F_{1,1} \rangle \cup \cdots \cup \langle F_{n,1} \rangle$, where $\langle F \rangle$ is the simplex spanned by $F$. Note that, if $u_{F_{i,j}} \in \alpha$, then $u_{F_{i,h}} \in \alpha$ for every $h \in \{ 1,\dots,m_i \}$, otherwise $\alpha$ does not have maximal length. In fact, if $u_{F_{i,j}} \in \alpha$, but $u_{F_{i,h}} \notin \alpha$ for some $h \neq j$, then the symbol $\beta$ obtained by adding $u_{F_{i,h}}$ to $\alpha$ is admissible and $|\beta| > |\alpha|$. Now suppose that $\alpha$ consists of the monomials:
$$
u_{h_1,j_1},\dots,u_{h_s,j_s} \text{ and } u_{F_{i_1,1}},\dots,u_{F_{i_1,m_{i_1}}}, \dots, u_{F_{i_k,1}},\dots,u_{F_{i_k,m_{i_k}}},
$$
ordered as above. We define the set
$$
A = \left\{ y^{(i_h,1)}_1,\dots,y^{(i_h,m_{i_h})}_1\ \Big|\ h=1,\dots,k \right\} \cup \left\{ x_i\ |\ i \in \{1,\dots,n \} \smallsetminus \{i_1,\dots,i_k\} \right\}.
$$
Then $A$ is a minimal vertex cover for $\Delta'$. First note that, for every $h \in \{1,\dots,n\}$, $x_h \in A$ or $y^{(i_h,1)}_1,\dots,y^{(i_h,m_{i_h})}_1 \in A$. Hence one of the vertices of each $F_{i,j}$ belongs to $A$. Moreover, if $y^{(j_1,1)}_1, \dots, y^{(j_t,1)}_1 \in A$, with $j_1,\dots,j_t$ pairwise distinct, then $x_{j_1} \cdots x_{j_t} \notin \Delta$. Otherwise, the monomial $x_{j_1} \cdots x_{j_t}$ would precede the monomials $u_{F_{(j_1,1)}}, \dots, u_{F_{(j_t,1)}}$ in the above ordering. Since $x_{j_1} \cdots x_{j_t}$ divides $u_{F_{(j_1,1)}} \cdots u_{F_{(j_t,1)}}$, this would imply that $\alpha$ is not admissible, against our assumption. This shows that one of the vertices of each facet of $\Delta$ belongs to $A$. Therefore $A$ is a vertex cover for $\Delta'$. It is minimal, because removing a vertex $y^{(i,j)}_1$ or a vertex $x_i$ would leave the facet $F_{i,j}$ uncovered.
\\ Finally, we prove that the length of $\alpha$ is less than or equal to $|A|$. Consider the symbol $\beta$ formed by the monomials
$$
u_{h_1,j_1},\dots,u_{h_s,j_s} \text{ and } u_{F_{i_1,1}},\dots,u_{F_{i_k,1}}.
$$
This is an admissible symbol on the vertex set $\left\{ x_1,\dots,x_n,y^{(1,1)}_1,\dots,y^{(1,1)}_{p_{(1,1)}},\dots, y^{(n,1)}_1,\dots, y^{(n,1)}_{p_{(n,1)}} \right\}$ for the simplicial complex $\Gamma$. Hence, by Proposition \ref{P2.whiskersimplcomplex}, $s+k \leq n$ and it follows that $|\alpha| = s + m_1 + \cdots + m_k \leq m_1 + \cdots + m_k + n-k = |A|$. Therefore,
\[
\bight I(\Delta') \leq \ara I(\Delta') \leq |\alpha| \leq |A| \leq \bight I(\Delta'),
\]
where the second inequality follows from Theorem \ref{Kimura}.
\end{proof}

\begin{rem}
Let $G$ be a graph and add $m_i \geq 1$ whiskers to each vertex of $G$. Let us call the graph $G'$ obtained in this way a \textit{multiwhisker graph on} $G$. Then, from Corollary \ref{multifacet}, it follows that $\bight I(G') = \ara I(G')$. Example \ref{multiwhisker} provides an explicit formula when $G=C_n$ and $m_i=k$ for every $i$.
\end{rem}

\bibliographystyle{plain}	                  

\end{document}